\documentclass[12pt,reqno]{amsart}

\usepackage{aliascnt,amsmath,amssymb,amsthm,amsfonts,mathtools,xcolor,shuffle}
\usepackage{CJKutf8}
\usepackage[marginparwidth=0pt,margin=20truemm]{geometry}

\definecolor{mylinkcolor}{RGB}{16,156,81}
\definecolor{mycitecolor}{RGB}{20,80,140}

\usepackage{hyperref}
\usepackage[nameinlink]{cleveref}
\hypersetup{
    bookmarksnumbered=true,
    bookmarksopen=true,
    colorlinks=true,
    linkcolor=mylinkcolor,
    citecolor=mycitecolor,
    urlcolor=mycitecolor,
    pdftitle={Shifted log-sine integrals and multiple zeta values},
    pdfauthor={Henrik Bachmann},
    pdfkeywords={multiple zeta values, shifted log-sine integrals},
}

\numberwithin{equation}{section}
\allowdisplaybreaks[2]

\theoremstyle{plain}
\newtheorem{thm}{Theorem}[section]
\newaliascnt{lem}{thm}
\newtheorem{lem}[lem]{Lemma}
\aliascntresetthe{lem}
\crefname{lem}{Lemma}{Lemmas}
\newaliascnt{prop}{thm}
\newtheorem{prop}[prop]{Proposition}
\aliascntresetthe{prop}
\crefname{prop}{Proposition}{Propositions}

\newtheorem{mainthm}{Main Theorem}

\crefname{mainthm}{Main Theorem}{Main Theorems}

\theoremstyle{definition}
\newaliascnt{rem}{thm}
\newtheorem{rem}[rem]{Remark}
\aliascntresetthe{rem}
\crefname{rem}{Remark}{Remarks}

\newcommand{\QQ}{\mathbb{Q}}
\newcommand{\cZ}{\mathcal{Z}}
\newcommand{\cS}{\mathcal{S}}
\newcommand{\cH}{\mathcal{H}}
\newcommand{\cA}{\mathcal{A}}
\newcommand{\fH}{\mathfrak{H}}
\newcommand{\cHall}{\cH^{\operatorname{MT}(\mathbb{Z}[\xi][1/6])}}
\newcommand{\sh}{\mathbin{\shuffle}}
\DeclareMathOperator{\SLs}{S}
\newcommand{\fM}{\mathfrak{M}}
\newcommand{\fN}{\mathfrak{N}}
\DeclareMathOperator{\LT}{LT}

\title{Shifted log-sine integrals and multiple zeta values}

\author{Henrik Bachmann}
\address{Graduate School of Mathematics, Nagoya University, Nagoya, Japan}
\email{henrik.bachmann@math.nagoya-u.ac.jp}

\date{\today}


\begin{document}

\begin{abstract}
We show that the shifted log-sine integrals with odd entries, introduced
by Umezawa, are multiple zeta values.  In particular, they define an
algebra homomorphism from the $f$-alphabet algebra to the
algebra of multiple zeta values.  The proof uses a motivic realization at
sixth roots of unity and the Galois descent from level six to level one
studied by Glanois.  Finally, we give an interpretation of our results
using Zinbiel algebras and introduce iterated cotangent integrals.
\end{abstract}

\maketitle

\section{Introduction}

Multiple zeta values are defined for integers $k_1,\dots,k_r\geq 1$
with $k_r\geq2$ by
\begin{align*}
 \zeta(k_1,\dots,k_r)
 \coloneqq
 \sum_{0<n_1<\cdots<n_r}
 \frac{1}{n_1^{k_1}\cdots n_r^{k_r}}.
\end{align*}
We call $k_1+\cdots+k_r$ the \emph{weight} and $r$ the \emph{depth} of
$\zeta(k_1,\dots,k_r)$.  Further, we set
$\zeta(\varnothing)=1$, which has weight $0$.  We denote by $\cZ$ the $\QQ$-vector space
spanned by all multiple zeta values and by $\cZ_k$ the space spanned by those
of weight $k$.  The space $\cZ$ is a $\QQ$-algebra.

Brown studied the algebra $\cH$ of motivic multiple zeta values $\zeta^{\mathfrak m}(k_1,\dots,k_r)$
\cite{Brown}.  It comes with a surjective algebra homomorphism
\begin{align*}
 \operatorname{per}:\cH
 &\longrightarrow\cZ,\\
 \zeta^{\mathfrak m}(k_1,\dots,k_r)
 &\longmapsto
 \zeta(k_1,\dots,k_r),
\end{align*}
called the period map.  Brown proved that there is a noncanonical graded
algebra isomorphism
\begin{align}\label{eq:f-alphabet}
 \cH
 \cong
 \QQ[f_2]\otimes
 \QQ\langle f_3,f_5,f_7,\dots\rangle_{\sh}.
\end{align}
The second factor is equipped with the shuffle product, and the letter
$f_k$ has weight $k$.  Conjecturally, the period map is injective and
thus an isomorphism.  Combining this with \eqref{eq:f-alphabet}, we therefore
expect a noncanonical algebra isomorphism
\begin{align}\label{eq:expected-isomorphism}
 \QQ[f_2]\otimes
 \QQ\langle f_3,f_5,f_7,\dots\rangle_{\sh}
 \overset{?}{\cong}\cZ.
\end{align}
Umezawa first studied iterated log-sine integrals in \cite{Umezawa-MZV},
generalizing the classical log-sine integrals (see, e.g.,
\cite{Borwein-Straub}), and later introduced shifted log-sine
integrals\footnote{Recently, Matsusaka \cite{Matsusaka} studied a
different, one-variable notion of shifted log-sine integrals.} in
\cite{Umezawa-evaluation}.  He conjectured that products of powers of
$\pi^2$ and these integrals with odd entries give a basis of the space
of multiple zeta values.  This family is indexed by words in odd integers and satisfies the
same shuffle product formula as the $f$-words, so it gives a natural
candidate for \eqref{eq:expected-isomorphism}.  However, it was not even
known if the shifted log-sine integrals are multiple zeta values, i.e., if
this candidate gives a well-defined map to $\cZ$.  In this work, we show
that this is the case.

Set $A(\theta)=\log|2\sin(\theta/2)|$.
For integers $k_1,\dots,k_r\geq2$, define the
\emph{shifted log-sine integral}
\begin{align*}
 \SLs(k_1,\dots,k_r)
 \coloneqq
 \int_{0<\theta_1<\cdots<\theta_r<\pi/3}
 \prod_{j=1}^r
 \left(\theta_j-\frac{\pi}{3}\right)^{k_j-2}A(\theta_j)\,
 d\theta_1\cdots d\theta_r.
\end{align*}
We set $\SLs(\varnothing)=1$.  In the following, we will only consider
odd entries.  For $k\geq0$, let
\begin{align}\label{eq:def-Sk}
 \cS_k=
 \operatorname{span}_{\QQ}
 \left\{
 \pi^{2m}\SLs(k_1,\dots,k_r)
 \mathrel{}\middle|\mathrel{}
 \begin{gathered}
 m,r\geq0,
 \quad k_1,\dots,k_r\geq3\text{ odd},\\
 2m+k_1+\cdots+k_r=k
 \end{gathered}
 \right\}.
\end{align}
Moreover, we set $\cS=\sum_{k\geq0}\cS_k$.
By Umezawa's shuffle product formula (see \eqref{eq:sls-shuffle}), the
space $\cS$ is a $\QQ$-algebra \cite[Proposition~4.1]{Umezawa-evaluation}.
Notice that this formula writes the product of two elements in
\eqref{eq:def-Sk} explicitly as a $\QQ$-linear combination of such
elements.  In contrast, by Brown's theorem the \emph{Hoffman elements},
i.e., the multiple zeta values with entries in $\{2,3\}$, span $\cZ$
\cite{Hoffman,Brown}, but no comparably simple formula is known for the
product of two Hoffman elements.

Umezawa conjectured that every multiple zeta value of weight $k$ and
depth $d$ lies in the span of the elements in \eqref{eq:def-Sk} with
$r\leq d$, and that the elements of weight $k$ form a basis of $\cZ_k$
\cite[Conjectures~2 and~3]{Umezawa-evaluation}.  Counting these elements
gives the generating series $1/(1-X^2-X^3)$, which is exactly the Hilbert
series of the algebra in \eqref{eq:f-alphabet}. The main result of this work is the following.

\begin{mainthm}\label{thm:main}
\begin{enumerate}
\item[\rm (i)] For odd $k_1,\dots,k_r\geq3$, we have
\begin{align*}
 \SLs(k_1,\dots,k_r)
 \in\cZ_{k_1+\cdots+k_r}.
\end{align*}
\item[\rm (ii)] In particular, the following is a $\QQ$-algebra
homomorphism.
\begin{align*}
 \QQ[f_2]\otimes
 \QQ\langle f_3,f_5,f_7,\dots\rangle_{\sh}
 &\longrightarrow \cZ,\\
 f_2^m\otimes f_{k_1}\cdots f_{k_r}
 &\longmapsto
 \pi^{2m}\SLs(k_1,\dots,k_r)
\end{align*}
\end{enumerate}
\end{mainthm}

Conjecturally, the homomorphism in part~(ii) gives an isomorphism as
in \eqref{eq:expected-isomorphism}, but both surjectivity and
injectivity are open problems.

Umezawa showed that shifted log-sine integrals can be written in terms of
multiple polylogarithms at sixth roots of unity \cite{Umezawa-evaluation}.
In Section~\ref{sec:words}, we represent the shifted log-sine integrals by
iterated integrals of words in a shuffle algebra
(\Cref{lem:word-integral}).  This representation works for arbitrary
entries $k_j\geq2$, and it gives a new proof that all shifted log-sine
integrals are $\QQ(i)$-linear combinations of multiple polylogarithms at
sixth roots of unity (\Cref{rem:level-six}~(ii)).  To prove
the main theorem, we combine it with the motivic Galois descent from level
six to level one developed by Glanois \cite{Glanois}.  We use the motivic
setup of Brown and Glanois as a black box, i.e., we only use some of their
structural results, which we recall in Section~\ref{sec:proof}.  After
this setup, we just have to check the descent condition.  This is done
by a downward induction, which uses a right-coideal property of these
words (\Cref{lem:right-coideal}).

In \Cref{sec:zinbiel}, we give a reinterpretation of the main theorem.
The elements representing the shifted log-sine integrals generate a
free Zinbiel algebra on two generators, which gives a model for the
$f$-alphabet algebra in \eqref{eq:f-alphabet}.  The main theorem then
says that an iterated cotangent integral gives an algebra homomorphism
from this model to $\cZ$, and Umezawa's conjecture says that it is an
isomorphism.  A similar model was proposed by Chapoton \cite{Chapoton}.

Note that the proof does not give an explicit expression for
$\SLs(k_1,\dots,k_r)$ in terms of multiple zeta values.  In
\Cref{sec:formulas}, we collect explicit evaluations in depth one,
going back to Choi, Cho and Srivastava \cite{Choi-Cho-Srivastava}, and
some conjectural evaluations in depths two and three.

\subsection*{Acknowledgments}
The author would like to thank Ryota Umezawa for comments on an
earlier version of this paper and for making him aware of the work
\cite{Kuramitsu}, and Fr\'ed\'eric Chapoton for his helpful comments
on \Cref{sec:zinbiel}.  This project was partially supported by JSPS
KAKENHI Grant Number JP26K22254.

\section{Words for shifted log-sine integrals}\label{sec:words}

Set $\xi=e^{\pi i/3}$, and let $\gamma$ be the arc from $1$ to $\xi$ given by
$z=e^{i\theta}$ for $0\leq\theta\leq\pi/3$.  Along $\gamma$, set
\begin{align*}
 x&=\frac{dz}{z},
 &
 y&=-\frac{dz}{1-z}-\frac12\frac{dz}{z},\\
 X(z)&=\log\frac{z}{\xi},
 &
 B(z)&=\log\frac{1-z}{1-\xi}-\frac12X(z).
\end{align*}
Here, we choose the branches of both logarithms such that they are
continuous along $\gamma$ and vanish at $z=\xi$.  Then we have
\begin{align}\label{eq:arc-functions}
 X(e^{i\theta})=i\left(\theta-\frac{\pi}{3}\right),
 \qquad
 B(e^{i\theta})=A(\theta),
 \qquad
 dX=x,
 \qquad
 dB=y.
\end{align}

Let $\fH=\QQ\langle x,y\rangle$ be the shuffle algebra on $x,y$, and let
$\fH_{>0}$ be the span of its nonempty words.  For
$p\geq1$, let $x^{\sh p}$ be the $p$-fold shuffle product of $x$ and define
\begin{align*}
 q_p=x^{\sh p}\sh y
 =p!\sum_{\substack{a,b\geq0\\a+b=p}}x^ayx^b.
\end{align*}
Define homogeneous elements recursively by
\begin{align}\label{eq:def-W}
 W_\varnothing=1,
 \qquad
 W_{p_1,\dots,p_r}
 =x\bigl(q_{p_1}\sh W_{p_2,\dots,p_r}\bigr).
\end{align}
The first product on the right-hand side of \eqref{eq:def-W} is the
concatenation product.  In \Cref{sec:zinbiel}, we give an
interpretation of these elements in terms of Zinbiel algebras.

Write $\gamma(t)=e^{\pi it/3}$ for $0\leq t\leq1$.  If
$u=a_1\cdots a_n$ is a word in $x,y$ and
$\gamma^*a_j=f_j(t)\,dt$, set
\begin{align*}
 I_\gamma(u)
 =
 \int_{0<t_1<\cdots<t_n<1}
 f_1(t_1)\cdots f_n(t_n)\,dt_1\cdots dt_n
\end{align*}
when this integral converges.
We set $I_\gamma(1)=1$ and extend $I_\gamma$ linearly whenever all
monomial integrals converge.  Every monomial word occurring in
$W_{p_1,\dots,p_r}$, for $r\geq1$, begins with $x$, which is regular at
$1$, and both $x$ and $y$ are regular at $\xi$.  Since $y$ has at worst a
simple pole at $z=1$, the standard convergence criterion for iterated
integrals shows that all iterated integrals used below are convergent.

\begin{lem}\label{lem:word-integral}
For odd $k_1,\dots,k_r\geq3$, we have
\begin{align}\label{eq:word-integral}
 \SLs(k_1,\dots,k_r)
 =(-1)^{\sum_{j=1}^r(k_j-1)/2}
 I_\gamma(W_{k_1-2,\dots,k_r-2}).
\end{align}
The element $W_{k_1-2,\dots,k_r-2}$ is homogeneous of length
$k_1+\cdots+k_r$.
\end{lem}

\begin{proof}
Let $I_\theta$ denote the iterated integral over the arc from
$e^{i\theta}$ to $\xi$, extended linearly in the same way as $I_\gamma$.  By \eqref{eq:arc-functions},
\begin{align*}
 I_\theta(x)=-X(e^{i\theta}),
 \qquad
 I_\theta(y)=-A(\theta).
\end{align*}
For odd $p$, the shuffle product formula gives
\begin{align*}
 I_\theta(q_p)
 =I_\theta(x)^pI_\theta(y)
 =X(e^{i\theta})^pA(\theta).
\end{align*}
Along the parametrization $z=e^{i\theta}$, the pullback of $x$ is
$i\,d\theta$.  The recursive step is
\begin{align*}
 I_\gamma(W_{p_1,\dots,p_r})
 &=\int_0^{\pi/3}
 I_\theta(q_{p_1}\sh W_{p_2,\dots,p_r})\,i\,d\theta\\
 &=\int_0^{\pi/3}
 X(e^{i\theta})^{p_1}A(\theta)
 I_\theta(W_{p_2,\dots,p_r})\,i\,d\theta.
\end{align*}
Iterating this identity gives
\begin{align*}
 I_\gamma(W_{p_1,\dots,p_r})
 =
 \int_{0<\theta_1<\cdots<\theta_r<\pi/3}
 \prod_{j=1}^r
 \left(X(e^{i\theta_j})^{p_j}A(\theta_j)i\,d\theta_j\right).
\end{align*}
If $p=k-2$ is odd, then
\begin{align*}
 X(e^{i\theta})^pA(\theta)i\,d\theta
 =(-1)^{(k-1)/2}
 \left(\theta-\frac{\pi}{3}\right)^{k-2}A(\theta)\,d\theta.
\end{align*}
This proves \eqref{eq:word-integral}.  Each operation
$x(q_p\sh\mathord{\cdot})$ adds $p+2=k$ letters.  Thus the element is
homogeneous of the stated length.
\end{proof}

For example, for $k=3$, we have $q_1=x\sh y=xy+yx$, and therefore
$W_1=x(xy+yx)=xxy+xyx$.  Since the pullbacks along $\gamma$ are given by
$x=i\,d\theta$ and
$y=\bigl(\frac{ie^{i\theta}}{e^{i\theta}-1}-\frac{i}{2}\bigr)d\theta
=\frac12\cot(\frac{\theta}{2})\,d\theta$,
\Cref{lem:word-integral} in this case reads
\begin{align*}
 \SLs(3)=-I_\gamma(xxy+xyx)
 =\frac12\int_{0<\theta_1<\theta_2<\theta_3<\pi/3}
 \left(\cot\frac{\theta_2}{2}+\cot\frac{\theta_3}{2}\right)
 d\theta_1\,d\theta_2\,d\theta_3.
\end{align*}
More generally, the same calculation gives, for every odd $k\geq3$,
\begin{align*}
 \SLs(k)=\frac{(k-2)!}{2}\sum_{j=2}^{k}\,
 \int_{0<\theta_1<\cdots<\theta_k<\pi/3}
 \cot\left(\frac{\theta_j}{2}\right)
 d\theta_1\cdots d\theta_k.
\end{align*}

\begin{rem}\label{rem:level-six}\leavevmode
\begin{enumerate}
\item[\rm (i)] A similar representation of the shifted log-sine
integrals, in terms of Yamamoto integrals, was obtained independently
by Kuramitsu in his master's thesis
\cite[Theorems~4.7 and~4.8]{Kuramitsu}.
\item[\rm (ii)] The definition \eqref{eq:def-W} and the proof of
\Cref{lem:word-integral} also work for arbitrary entries
$k_1,\dots,k_r\geq2$, where we set $q_0=y$.  In general, one obtains
\begin{align*}
 I_\gamma(W_{k_1-2,\dots,k_r-2})
 =(-i)^{k_1+\cdots+k_r-r}\,
 \SLs(k_1,\dots,k_r).
\end{align*}
By the change of variable in the proof of \Cref{prop:membership}, the
left-hand side is a $\QQ(i)$-linear combination of multiple polylogarithms
at sixth roots of unity (see \cite{Borwein-Broadhurst-Kamnitzer} for a
classical study of these values).  This gives a new proof of a result of
Umezawa \cite{Umezawa-evaluation}, who first showed that all shifted
log-sine integrals can be written in this way.
\end{enumerate}
\end{rem}

Next, we prove the right-coideal property of the elements $W_{\geq j}$, which
will be a key ingredient in the proof of the main theorem.  Equip
$\fH$ with the deconcatenation coproduct
\begin{align*}
 \Delta_{\mathrm{dec}}(a_1\cdots a_N)
 =\sum_{j=0}^N
 a_1\cdots a_j\otimes a_{j+1}\cdots a_N.
\end{align*}
Fix $p_1,\dots,p_r$ and set $W_{\geq j}=W_{p_j,\dots,p_r}$ for
$1\leq j\leq r$ as well as $W_{\geq r+1}=1$.  For $1\leq j\leq r+1$,
let $\mathfrak B_j$ be the shuffle subalgebra generated by $x$, $y$, and
the elements $W_{\geq i}$ for $j\leq i\leq r$.  In particular,
$\mathfrak B_{r+1}$ is generated by $x$ and $y$.

\begin{lem}\label{lem:right-coideal}
For $1\leq j\leq r$, we have
\begin{align}\label{eq:right-coideal}
 \Delta_{\mathrm{dec}}W_{\geq j}-1\otimes W_{\geq j}
 \in
 \fH_{>0}\otimes\mathfrak B_{j+1}.
\end{align}
In particular, every $\mathfrak B_j$ is a right coideal subalgebra,
i.e., we have $\Delta_{\mathrm{dec}}\mathfrak B_j\subset\fH\otimes\mathfrak B_j$.
\end{lem}

\begin{proof}
The coproduct $\Delta_{\mathrm{dec}}$ is an algebra homomorphism for the
shuffle product.  Since $x$ and $y$ are primitive, we have
\begin{align}\label{eq:coproduct-qp}
 \Delta_{\mathrm{dec}}q_p
 =
 (1\otimes x+x\otimes1)^{\sh p}
 \sh(1\otimes y+y\otimes1)
 \in
 \fH\otimes\mathfrak B_{r+1}.
\end{align}

For $P(F)=xF$, where the product is concatenation, we have
\begin{align}\label{eq:prefix-coproduct}
 \Delta_{\mathrm{dec}}P(F)
 =1\otimes P(F)+(P\otimes\operatorname{id})
 \Delta_{\mathrm{dec}}F.
\end{align}
Apply \eqref{eq:prefix-coproduct} to
$F=q_{p_j}\sh W_{\geq j+1}$, where the case $j=r$ is
$F=q_{p_r}$.  By \eqref{eq:coproduct-qp} and downward
induction,
\begin{align*}
 \Delta_{\mathrm{dec}}F
 \in
 \fH\otimes\mathfrak B_{j+1}.
\end{align*}
Equation \eqref{eq:prefix-coproduct} now gives
\eqref{eq:right-coideal}.  Applying the same argument to the generators of
$\mathfrak B_j$ gives the last assertion.
\end{proof}

\section{Proof of the main theorem}\label{sec:proof}

The motivic fundamental group at roots of unity was introduced by Deligne and
Goncharov \cite{Deligne-Goncharov}.  The level-six case was studied by
Deligne \cite{Deligne}, and its Galois descent to level one was developed by
Glanois \cite{Glanois}.  We recall the results which we need.  Let $\cHall$ be
the algebra of motivic periods of the category
$\operatorname{MT}(\mathbb{Z}[\xi][1/6])$ of \cite{Deligne-Goncharov} (see
\cite[(1.5)]{Glanois}).  It contains all motivic multiple zeta values relative
to $\mu_6$.  Let $\cH_6\subset\cHall$ be the corresponding algebra for the
category $\operatorname{MT}(\mathbb{Z}[\xi])$.  It contains the unramified
motivic multiple zeta values relative to $\mu_6$
(cf. \cite[Section~4.4.4]{Glanois}).  Let $\cA_6$ and $\cA$ be the coordinate
rings of the unipotent parts of the motivic Galois groups of
$\operatorname{MT}(\mathbb{Z}[\xi])$ and $\operatorname{MT}(\mathbb{Z})$.  The
motivic coaction $\Delta:\cH_6\rightarrow\cA_6\otimes\cH_6$ is the
restriction of the corresponding coaction on $\cHall$, and the algebras
$\cH$ and $\cA$ are subalgebras of $\cH_6$ and $\cA_6$.

We use the following standard form of the results of Brown and Glanois as a
black box.  The Lie algebra of the unipotent part of the motivic Galois group
is free with one generator in every odd weight $m\geq3$ at level one and
with one generator in every weight $m\geq2$ at level six
(see \cite[Section~2.5]{Brown} and \cite[Lemma~2.1]{Glanois}).  Together
with Brown's theorem and the Galois descent from level six to level one
\cite[Sections~3.3 and~4.4.4]{Glanois}, this gives noncanonical graded
comodule algebra isomorphisms
\begin{equation}\label{eq:level-six-alphabet}
\begin{aligned}
 \cH_6
 &\cong
 \QQ\langle g_2,g_3,g_4,\dots\rangle_{\sh}\otimes\QQ[t],\\
 \cH
 &\cong
 \QQ\langle g_3,g_5,g_7,\dots\rangle_{\sh}\otimes\QQ[t^2],
\end{aligned}
\end{equation}
which are compatible with the inclusion $\cH\subset\cH_6$.  We use $g_m$
to distinguish this level-six alphabet from the $f$-alphabet in the
introduction.  The letter $g_m$ has weight $m$, the element $t$ has weight $1$ and
corresponds to $(2\pi i)^{\mathfrak m}$, and the coaction is
deconcatenation on the $g$-words, where $t$ stays in the right factor.  Explicitly,
\begin{align*}
 \Delta(g_{i_1}\cdots g_{i_r}t^s)
 =\sum_{a=0}^r
 g_{i_1}\cdots g_{i_a}
 \otimes g_{i_{a+1}}\cdots g_{i_r}t^s,
\end{align*}
where the empty word is $1$.  In other words, the odd letters are the
level-one letters and the even letters are the new level-six letters.  Complex
conjugation acts on $\cHall$ by conjugating the path, the endpoints and the
letters of a motivic iterated integral (see \cite[Appendix~A]{Deligne}).  We
denote this involution by $c$.  It
fixes $\cH$ pointwise and satisfies $c(t)=-t$.  Since the generators of the
free Lie algebras above can be chosen as eigenvectors of $c$, the
isomorphisms in \eqref{eq:level-six-alphabet} can moreover be chosen such that
$c$ acts on $\cA_6\cong\QQ\langle g_2,g_3,\dots\rangle_{\sh}$ letterwise by
\begin{align}\label{eq:conjugation-alphabet}
 c(g_{i_1}\cdots g_{i_r})
 =(-1)^{(i_1-1)+\cdots+(i_r-1)}\,g_{i_1}\cdots g_{i_r}.
\end{align}
Here the signs are given by the depth-one conjugation relation
$\zeta^{\mathfrak a}(m;\xi^{-1})=(-1)^{m-1}\zeta^{\mathfrak a}(m;\xi)$ of
Deligne and Goncharov (see \cite[Lemmas~3.1 and~3.2]{Glanois}).

For $Z\in\cHall$, denote by $Z^{\mathfrak a}$ its de Rham image, i.e., its
image in the coordinate ring of the unipotent part of the motivic Galois group
of $\operatorname{MT}(\mathbb{Z}[\xi][1/6])$, which contains $\cA_6$.  On
$\cH_6$, this is the natural projection $\cH_6\rightarrow\cA_6$, which in
\eqref{eq:level-six-alphabet} is given by $t\mapsto0$.  We use the same
superscript for the images of motivic iterated integrals.  Set
\begin{align*}
 \widetilde\Delta Z
 =\Delta Z-1\otimes Z-Z^{\mathfrak a}\otimes1.
\end{align*}
We write $\cA_{>0}$ and $\cH_{6,>0}$ for the positive-weight parts of
$\cA$ and $\cH_6$.  The condition in the following lemma implies that all proper prefixes
of a $g$-word consist of odd letters.  Besides the elements of $\cH$,
this just leaves words ending in one even letter and odd powers of
$t$.  Both are then excluded by complex conjugation.

\begin{lem}\label{lem:full-descent}
Let $Z\in\cH_6$ be homogeneous of positive weight and fixed by complex
conjugation.  Suppose that
\begin{align*}
 \widetilde\Delta Z
 \in
 \cA_{>0}\otimes\cH_{6,>0}.
\end{align*}
Then $Z\in\cH$.
\end{lem}

\begin{proof}
By \eqref{eq:level-six-alphabet}, we can write
\begin{align*}
 Z=\sum_{s\geq0}Z_s t^s,
 \qquad
 Z_s\in\QQ\langle g_2,g_3,g_4,\dots\rangle_{\sh}.
\end{align*}
Set
\begin{align*}
 V=\QQ\langle g_3,g_5,g_7,\dots\rangle_{\sh}.
\end{align*}
If $s>0$, projecting the hypothesis to the right factor $t^s$
gives
\begin{align*}
 Z_s-\varepsilon(Z_s)\in V,
\end{align*}
where $\varepsilon$ is the counit.  Since the constant part also belongs
to $V$, we have $Z_s\in V$.

For $s=0$, write uniquely
\begin{align*}
 Z_0=\lambda+\sum_{m\geq2}A_mg_m,
\end{align*}
where the last product is concatenation.  Projecting the hypothesis to
the right factor $g_m$ shows that the
nonconstant part of $A_m$ belongs to $V$.  Its constant part already
belongs to $V$, and we get
\begin{align*}
 Z_0\in V+\sum_{m\geq2\text{ even}}Vg_m.
\end{align*}
The projection $\cH_6\rightarrow\cA_6$ commutes with $c$, so $Z_0=Z^{\mathfrak a}$
is fixed by $c$.  By \eqref{eq:conjugation-alphabet}, $c$ fixes $V$ and
changes the sign of every term $Vg_m$ with $m$ even, so $Z_0\in V$.
We obtain $Z\in V\otimes\QQ[t]$, which by \eqref{eq:level-six-alphabet} equals
$\cH+t\cH$.  Write $Z=Z'+tZ''$ with $Z',Z''\in\cH$.  Since $c$ fixes $\cH$
and $c(t)=-t$, applying $c$ gives $Z''=0$, i.e., $Z\in\cH$.
\end{proof}

\begin{prop}\label{prop:membership}
For odd $k_1,\dots,k_r\geq3$, we have
\begin{align*}
 \SLs(k_1,\dots,k_r)
 \in\cZ_{k_1+\cdots+k_r}.
\end{align*}
\end{prop}

\begin{proof}
For $a\in\mathbb C$, set $e_a=dz/(z-a)$.  Then
\begin{align*}
 x=e_0,
 \qquad
 y=e_1-\frac12e_0.
\end{align*}
We expand a word in $x,y$ in the letters $e_0,e_1$ and define its motivic
iterated integral term by term.  We consider these iterated integrals on
$\mathbb{P}^1\setminus(\{0,\infty\}\cup\mu_6)$, and, since the endpoints
$1$ and $\xi$ of $\gamma$ lie in $\mu_6$, we use tangential base points
as in \cite[Section~4.3]{Deligne-Goncharov}.  All tangent vectors below are
roots of unity, and for de Rham images their choice is irrelevant: for a
root of unity $\eta$, the logarithm $\log(\eta)$ is a rational multiple of
$2\pi i$, so $\log^{\mathfrak m}(\eta)$ is a rational multiple of the
element $t$ and $\log^{\mathfrak a}(\eta)=0$
\cite[Section~5.4]{Deligne-Goncharov}.

Set $p_j=k_j-2$.  Every monomial word
occurring in $W_{\geq j}$ begins with $x$, which is regular at $1$, and both $x$
and $y$ are regular at $\xi$.  In particular, the motivic iterated integrals
\begin{align*}
 Z_j=I^{\mathfrak m}_\gamma(W_{\geq j})
\end{align*}
are convergent and do not depend on the tangential base points.
The change of variable
\begin{align*}
 u=\frac{z-1}{\xi-1}
\end{align*}
turns the path from $1$ to $\xi$ into a path from $0$ to $1$ and gives
\begin{align*}
 x=\frac{du}{u-\xi},
 \qquad
 y=\frac{du}{u}-\frac12\frac{du}{u-\xi}.
\end{align*}
Since $\xi-1=\xi^2$, we have
\begin{align*}
 z=1+\xi^2u=\xi^2(u-\xi).
\end{align*}
This shows that the image of $\gamma$ is the lower short arc of
$|u-\xi|=1$ from $0$ to $1$.  The region between this arc and the interval
$[0,1]$ contains no sixth root of unity, so the arc is homotopic to
the straight path in $\mathbb{P}^1\setminus(\{0,\infty\}\cup\mu_6)$.
Every transformed monomial word starts with $du/(u-\xi)$ and contains
only the letters $0$ and $\xi$.  After grouping consecutive zeros into blocks, its motivic iterated integral has the form
\begin{align*}
 I^{\mathfrak m}
 (0;\xi,0^{n_1-1},\dots,\xi,0^{n_d-1};1)
 =(-1)^d\zeta^{\mathfrak m}
 \binom{n_1,\dots,n_d}{1,\dots,1,\xi^{-1}},
\end{align*}
where $d\geq1$, $n_1,\dots,n_d\geq1$, and $0^n$ denotes $n$ repeated
zeroes.  The right-hand side is a motivic
multiple zeta value relative to $\mu_6$ in Glanois's notation.  In this
notation, the associated roots of unity $\eta_1,\dots,\eta_d$ are
exactly the nonzero letters of the transformed words, so they are all
equal to $\xi$.  By Glanois's unramified criterion, $Z_j$ therefore
belongs to $\cH_6$ \cite[Definition~2.2 and Lemma~4.14]{Glanois}.

We prove $Z_j\in\cH$ by downward induction on $j$.  Assume that this is
known for $Z_{j+1},\dots,Z_r$.  We first prove the coaction part of the
induction step,
\begin{align}\label{eq:coaction-Zj}
 \widetilde\Delta Z_j
 \in
 \cA_{>0}\otimes\cH_{6,>0}.
\end{align}
Let $F\in\mathfrak B_{j+1}$.  The evaluation
$F\mapsto I^{\mathfrak a}(1;F;\xi)$ along $\gamma$ is an algebra homomorphism for the
shuffle product.  By induction, the evaluations
of the elements $W_{\geq i}$ for $i>j$ are
$Z_{j+1}^{\mathfrak a},\dots,Z_r^{\mathfrak a}$ and belong to $\cA$.
Using $\xi-1=\xi^2$, we have
\begin{align*}
 I^{\mathfrak a}(1;x;\xi)
 &=\log^{\mathfrak a}(\xi)=0,\\
 I^{\mathfrak a}(1;y;\xi)
 &=\log^{\mathfrak a}(\xi-1)
 -\frac12\log^{\mathfrak a}(\xi)=0.
\end{align*}
This gives
\begin{align}\label{eq:evaluation-one}
 I^{\mathfrak a}(1;F;\xi)\in\cA.
\end{align}
For the path from $0$ to $\xi$, regularized path composition through $1$
(\cite[Section~2, property~(iv)]{Glanois}, with the empty-word terms
included) gives
\begin{align}\label{eq:path-composition}
 I^{\mathfrak a}(0;F;\xi)
 =\sum_{\Delta_{\mathrm{dec}}F}
 I^{\mathfrak a}(0;F_{(1)};1)
 I^{\mathfrak a}(1;F_{(2)};\xi).
\end{align}
By \Cref{lem:right-coideal}, we may choose the decomposition
$\Delta_{\mathrm{dec}}F=\sum F_{(1)}\otimes F_{(2)}$ with
$F_{(2)}\in\mathfrak B_{j+1}$.  The first factor in
\eqref{eq:path-composition} has endpoints and letters in $\{0,1\}$ and
belongs to $\cA$, and the second factor belongs to $\cA$ by
\eqref{eq:evaluation-one}.  We obtain
\begin{align}\label{eq:evaluation-terminal}
 I^{\mathfrak a}(a;F;\xi)\in\cA
 \qquad
 (a\in\{0,1\},\ F\in\mathfrak B_{j+1}).
\end{align}

We now apply the full coaction.  For a monomial word
$e_{a_1}\cdots e_{a_n}$ in the expansion of $W_{\geq j}$, set $a_0=1$ and
$a_{n+1}=\xi$.  By Goncharov's formula \cite{Goncharov}, with tangential
regularization as in \cite[Theorem~3.3]{Glanois}, we have
\begin{equation}\label{eq:goncharov}
\begin{aligned}
 &\Delta I^{\mathfrak m}(a_0;a_1,\dots,a_n;a_{n+1})\\
 &\qquad=\sum_{0=i_0<i_1<\cdots<i_k<i_{k+1}=n+1}
 \ \prod_{p=0}^{k}
 I^{\mathfrak a}(a_{i_p};a_{i_p+1},\dots,a_{i_{p+1}-1};a_{i_{p+1}})
 \otimes
 I^{\mathfrak m}(a_0;a_{i_1},\dots,a_{i_k};a_{n+1}),
\end{aligned}
\end{equation}
where the sum runs over all subsets $\{i_1<\cdots<i_k\}$ of
$\{1,\dots,n\}$, and where the right factor is taken along $\gamma$.  We
apply \eqref{eq:goncharov} linearly to $W_{\geq j}$.  The empty subset gives
$Z_j^{\mathfrak a}\otimes1$, and the full subset gives $1\otimes Z_j$.  In
all other terms both factors have positive weight, and we show that their
left factors lie in $\cA$.  For this, fix the last selected position
$s=i_k\in\{1,\dots,n\}$.  All factors except the last de Rham factor
$I^{\mathfrak a}(a_s;a_{s+1},\dots,a_n;\xi)$ then depend only on the prefix
$e_{a_1}\cdots e_{a_s}$.  So the total contribution of the terms with last
selected position $s$ is the image of
$(\Delta_{\mathrm{dec}}W_{\geq j})_{s,n-s}\in\fH_s\otimes\fH_{n-s}$
under the linear map given on monomials by
\begin{align*}
 e_{a_1}\cdots e_{a_s}\otimes w
 \longmapsto
 \sum_{i_k=s}
 \ \prod_{p=0}^{k-1}
 I^{\mathfrak a}(a_{i_p};a_{i_p+1},\dots,a_{i_{p+1}-1};a_{i_{p+1}})
 \, I^{\mathfrak a}(a_s;w;\xi)
 \otimes
 I^{\mathfrak m}(1;a_{i_1},\dots,a_{i_k};\xi),
\end{align*}
where the sum runs over all subsets $\{i_1<\cdots<i_k\}$ of $\{1,\dots,s\}$ containing $s$.
Here $\fH_s$ denotes the length-$s$ part of $\fH$.  For a monomial
$e_{a_1}\cdots e_{a_s}$ and $F\in\mathfrak B_{j+1}$, the image of
$e_{a_1}\cdots e_{a_s}\otimes F$ lies in $\cA\otimes\cHall$, since the
de Rham factors in the product have endpoints and letters in $\{0,1\}$, and
since $I^{\mathfrak a}(a_s;F;\xi)\in\cA$ by \eqref{eq:evaluation-terminal}.
As $(\Delta_{\mathrm{dec}}W_{\geq j})_{s,n-s}\in\fH_s\otimes\mathfrak B_{j+1}$
by \eqref{eq:right-coideal}, we obtain
\begin{align*}
 \widetilde\Delta Z_j
 \in
 \cA_{>0}\otimes\cHall_{>0}.
\end{align*}
Since $Z_j\in\cH_6$, we also have $\widetilde\Delta Z_j\in\cA_6\otimes\cH_6$,
and \eqref{eq:coaction-Zj} follows.

It remains to check that $Z_j$ is fixed by complex conjugation.  The
letters $e_0,e_1$ are real, and the conjugate path of $\gamma$ is
$h\circ\gamma$, where $h(z)=1/z$.  Directly from the definitions,
\begin{align*}
 h^*x=-x,
 \qquad
 h^*y=y.
\end{align*}
Since $p_j$ is odd, this gives $h^*q_{p_j}=-q_{p_j}$.  The two signs in
\eqref{eq:def-W} cancel at each step, and therefore
$h^*W_{\geq j}=W_{\geq j}$.  As $Z_j$ is convergent, the tangential base
points play no role, and functoriality gives
\begin{align*}
 c(Z_j)
 =I^{\mathfrak m}_{h\circ\gamma}(W_{\geq j})
 =I^{\mathfrak m}_{\gamma}(h^*W_{\geq j})
 =Z_j.
\end{align*}

By \Cref{lem:word-integral}, $W_{\geq j}$ is homogeneous of length
$k_j+\cdots+k_r$.  So $Z_j$ is homogeneous of this weight, and
\Cref{lem:full-descent} applied to \eqref{eq:coaction-Zj} gives
$Z_j\in\cH$.  This completes the induction.  The period of $Z_j$ is
$\SLs(k_j,\dots,k_r)$ up to a sign, and taking $j=1$ proves the proposition.
\end{proof}

\begin{proof}[Proof of \Cref{thm:main}]
Part~(i) is \Cref{prop:membership}.  For part~(ii), Umezawa's shuffle product
formula, with shuffle multiplicities kept as in
\cite[Proposition~4.1]{Umezawa-evaluation}, gives
\begin{align}\label{eq:sls-shuffle}
 \SLs(k_1,\dots,k_r)\SLs(l_1,\dots,l_s)
 =\sum_{\mathbf a\in
 (k_1,\dots,k_r)\sh(l_1,\dots,l_s)}
 \SLs(\mathbf a).
\end{align}
The assignment in part~(ii) defines a $\QQ$-linear map on the basis
given by the elements $f_2^m\otimes f_{k_1}\cdots f_{k_r}$.  By
part~(i), it takes values in $\cZ$.  Equation
\eqref{eq:sls-shuffle} shows that it respects the shuffle product, and
powers of $f_2$ are sent to the corresponding powers of $\pi^2$.  This proves
part~(ii).
\end{proof}

\section{Zinbiel algebras and iterated cotangent integrals}\label{sec:zinbiel}

In this section, we give a reinterpretation of \Cref{thm:main} and,
inspired by \Cref{lem:word-integral}, introduce iterated cotangent
integrals.  Zinbiel algebras were studied by Chapoton \cite{Chapoton}
in connection with multiple zeta values, which is also how the author
became aware of them.

A \emph{Zinbiel algebra}\footnote{The curious reader wondering where
the name Zinbiel comes from should read it backwards.} (Loday
\cite{Loday}) is a $\QQ$-vector space
with a bilinear product $\prec$ satisfying
\begin{align}\label{eq:zinbiel}
 (u\prec v)\prec w=u\prec(v\prec w)+u\prec(w\prec v)
\end{align}
for all elements $u,v,w$.  The symmetrized product $u\prec v+v\prec u$
is then commutative and associative.  On the space $\fH_{>0}$, a
Zinbiel product is defined by
\begin{align}\label{eq:half-shuffle}
 (au)\prec v\coloneqq a(u\sh v)
\end{align}
for a letter $a\in\{x,y\}$ and words $u,v$, and its symmetrization is
the shuffle product.  With this product, $\fH_{>0}$ is the free
Zinbiel algebra on $\{x,y\}$ (see \cite{Loday,Chapoton}).  When
adjoining the empty word $1$ as the unit for the shuffle product, we
use the conventions $u\prec1=u$ and $1\prec u=0$ for $u\in\fH_{>0}$.  By
\eqref{eq:zinbiel}, the Zinbiel subalgebra generated by two elements
$c_1,c_2\in\fH_{>0}$ is spanned by the right-nested products
$c_{j_1}\prec(c_{j_2}\prec(\cdots\prec c_{j_N}))$ with
$j_1,\dots,j_N\in\{1,2\}$.

Let $\fM_{>0}$ be the Zinbiel subalgebra generated by $xx$ and
$W_1=xxy+xyx$, and set $\fM=\QQ\oplus\fM_{>0}$.  Since the shuffle
product is the symmetrization of $\prec$, the space $\fM$ is a
subalgebra of $(\fH,\sh)$.

\begin{prop}\label{prop:zinbiel}\leavevmode
\begin{enumerate}
\item[\rm (i)] $\fM_{>0}$ is the free Zinbiel algebra on the two elements
$xx$ and $W_1$.
\item[\rm (ii)] The elements $(xx)^{\sh m}\sh W_{p_1,\dots,p_r}$ with
$m,r\geq0$ and odd $p_j\geq1$ form a $\QQ$-basis of $\fM$.
\item[\rm (iii)] The following is an algebra isomorphism.
\begin{align*}
 \QQ[f_2]\otimes\QQ\langle f_3,f_5,f_7,\dots\rangle_{\sh}
 &\overset{\sim}{\longrightarrow}\fM\\
 f_2^m\otimes f_{k_1}\cdots f_{k_r}
 &\longmapsto(xx)^{\sh m}\sh W_{k_1-2,\dots,k_r-2}
\end{align*}
\end{enumerate}
\end{prop}

\begin{proof}
Since $W_{p_1}=x\,q_{p_1}$, the definition \eqref{eq:def-W} can be
written as
\begin{align}\label{eq:W-prec}
 W_{p_1,\dots,p_r}=W_{p_1}\prec W_{p_2,\dots,p_r}.
\end{align}
Moreover, $x^{\sh2}=2\,xx$ gives $q_{p+2}=2\,(q_p\sh xx)$, and
therefore
\begin{align}\label{eq:W-step}
 W_{p+2}=2\,(W_p\prec xx),
 \qquad
 W_{2r+1}=2^{r}\,\bigl(W_1\prec(xx)^{\sh r}\bigr).
\end{align}
So all $W_{p_1,\dots,p_r}$ with odd entries lie in $\fM_{>0}$, and hence
all elements $(xx)^{\sh m}\sh W_{p_1,\dots,p_r}$ lie in $\fM$.

Order the words of a fixed length lexicographically with $y>x$, and
write $\LT(F)$ for the largest word of a nonzero homogeneous
$F\in\fH$.  Then the largest word of $x\sh F$ is $\LT(F)x$, and the
largest word of $y\sh F$ is $y\LT(F)$.  For $F$ with all words
beginning with $x$, we get
\begin{align*}
 \LT(xx\prec F)=x\LT(F)x
 \qquad\text{and}\qquad
 \LT(W_1\prec F)=xy\LT(F)x,
\end{align*}
since all words of $xy\sh F$ begin with $x$ and are therefore smaller
than the largest word $y\LT(F)x$ of $yx\sh F$.  Both operations
preserve the properties that all words begin with $x$ and that the
leading coefficient is positive.  By induction, we get
\begin{align}\label{eq:leading-word}
 \LT\bigl(c_1\prec(c_2\prec(\cdots\prec c_N))\bigr)
 =x^{i_1}y\,x^{i_2-i_1}y\cdots x^{i_s-i_{s-1}}y\,x^{2N-i_s}
\end{align}
for $c_j\in\{xx,W_1\}$, where $i_1<\cdots<i_s$ are the positions with
$c_j=W_1$, and where the right-hand side is understood as $x^{2N}$ if
$s=0$.  The right-hand side of \eqref{eq:leading-word} determines
$(c_1,\dots,c_N)$, and therefore the right-nested products are
linearly independent.  This proves part (i).

The free Zinbiel algebra on two letters $a,b$ has the nonempty words
in $a,b$ as a basis.  For $r\geq0$, set
$B_r=b\prec a^{\sh r}=r!\,b\,a^r$.  For the
lexicographic order with $a>b$, the same induction as above gives
\begin{align*}
 \LT\bigl(a^{\sh m}\sh(B_{r_1}\prec(\cdots\prec B_{r_s}))\bigr)
 =a^m\,b\,a^{r_1}\,b\,a^{r_2}\cdots b\,a^{r_s},
\end{align*}
where for $s=0$ the left-hand side is understood as $a^{\sh m}$ and
the right-hand side as $a^m$.  Every word in $a,b$ is of this form
for exactly one $(m,r_1,\dots,r_s)$, so these elements, together
with $1$, form a basis after adjoining the unit for the shuffle
product.  The Zinbiel
homomorphism $a\mapsto xx$, $b\mapsto W_1$ also respects the shuffle
products and is injective by part (i).  By \eqref{eq:W-prec} and
\eqref{eq:W-step}, it sends $B_{r_1}\prec(\cdots\prec B_{r_s})$ to a
positive rational multiple of $W_{2r_1+1,\dots,2r_s+1}$.  Its image is $\fM_{>0}$, so the elements
$(xx)^{\sh m}\sh W_{p_1,\dots,p_r}$ with odd entries form a basis of
$\fM$.  This proves part (ii).

For (iii) and nonempty index tuples $\mathbf p,\mathbf q$,
symmetrizing \eqref{eq:W-prec} and using \eqref{eq:zinbiel}
gives
\begin{align}\label{eq:index-shuffle}
 W_{\mathbf p}\sh W_{\mathbf q}
 =W_{p_1}\prec\bigl(W_{p_2,\dots,p_r}\sh W_{\mathbf q}\bigr)
 +W_{q_1}\prec\bigl(W_{\mathbf p}\sh W_{q_2,\dots,q_s}\bigr)
 =\sum_{\mathbf r\in\mathbf p\,\sh\,\mathbf q}
 W_{\mathbf r},
\end{align}
where the last equality follows by induction and the sum runs over all
shuffles of the two index tuples.  So the map in (iii) is an algebra
homomorphism, and by part (ii) it sends a basis to a basis.
\end{proof}

\begin{rem}
By \Cref{prop:zinbiel}~(i), the right-nested products
\begin{align*}
 K(c_1\cdots c_N)
 =c_1\prec\bigl(c_2\prec(\cdots\prec c_N)\bigr),
 \qquad
 c_j\in\{xx,W_1\},
\end{align*}
form a basis of $\fM_{>0}$, indexed by the words $c_1\cdots c_N$ in
the two generators.  The same induction as in
\eqref{eq:index-shuffle} gives
\begin{align*}
 K(\mathbf c)\sh K(\mathbf d)
 =\sum_{\mathbf e\in\mathbf c\,\sh\,\mathbf d}K(\mathbf e),
\end{align*}
where the sum runs over all shuffles of the two indexing words.
Chapoton uses the corresponding basis in his construction
\cite{Chapoton}.
\end{rem}

For a word $w=a_1\cdots a_n$ beginning with the letter $a_1=x$, define
the \emph{iterated cotangent integral}
\begin{align*}
 C(w)
 \coloneqq
 \int_{0<t_1<\cdots<t_n<\pi/6}
 \prod_{j=1}^n g_{a_j}(t_j)\,dt_j,
 \qquad
 g_x(t)=1,
 \quad
 g_y(t)=\cot(t),
\end{align*}
and set $C(1)=1$.  Since the first letter is $x$ and $\cot$ has
only a simple pole at $0$, these integrals converge, and we extend $C$
to a $\QQ$-linear map on $\QQ+x\fH$.  This space is closed under the
shuffle product, and iterated integrals satisfy the shuffle product
formula, so $C$ is an algebra homomorphism.  For example, we have
$C(xx)=\frac{1}{12}\zeta(2)=\frac{\pi^2}{72}$.  The pullbacks along
$\gamma$ are given by
$x=i\,d\theta$ and $y=\frac12\cot(\frac{\theta}{2})\,d\theta$, so the
substitution $\theta=2t$ gives $I_\gamma(w)=(2i)^{d}\,C(w)$ for a word
$w$ with $d$ letters $x$.  Therefore, \Cref{rem:level-six}~(ii)
becomes
\begin{align}\label{eq:S-as-C}
 \SLs(k_1,\dots,k_r)
 =(-2)^{k_1+\cdots+k_r-r}\,
 C(W_{k_1-2,\dots,k_r-2})
\end{align}
for all $k_1,\dots,k_r\geq2$.  For example, we have $\SLs(3)=4\,C(W_1)$,
and therefore $C(W_1)=\frac16\zeta(3)$ by \Cref{prop:depth-one} below.
So the two generators of $\fM_{>0}$ are sent to $\frac{1}{12}\zeta(2)$
and $\frac16\zeta(3)$.  Notice that applying \eqref{eq:S-as-C} to
\eqref{eq:index-shuffle} gives a new proof of Umezawa's shuffle
product formula \eqref{eq:sls-shuffle}.  With this, we can reformulate
\Cref{thm:main} as follows.

\begin{thm}\label{thm:cotangent}
The iterated cotangent integral restricts to a $\QQ$-algebra
homomorphism
\begin{align*}
 C\colon\fM\longrightarrow\cZ,
\end{align*}
which sends homogeneous elements of length $k$ into $\cZ_k$.
\end{thm}

\begin{proof}
The map $C$ is an algebra homomorphism on $\QQ+x\fH$, and
$\fM\subset\QQ+x\fH$.  By \Cref{prop:zinbiel}, it therefore suffices
to check the claim on the basis elements.  By \eqref{eq:S-as-C} and
\Cref{thm:main}, we have
\begin{align*}
 C\bigl((xx)^{\sh m}\sh W_{k_1-2,\dots,k_r-2}\bigr)
 =\left(\frac{\pi^2}{72}\right)^{m}
 \frac{\SLs(k_1,\dots,k_r)}{(-2)^{k_1+\cdots+k_r-r}},
\end{align*}
which lies in $\cZ_{2m+k_1+\cdots+k_r}$.
\end{proof}

Composing the isomorphism in \Cref{prop:zinbiel}~(iii) with $C$ gives
the homomorphism in \Cref{thm:main}~(ii), up to the rational factors
$72^{m}(-2)^{k_1+\cdots+k_r-r}$ on the basis elements.  In particular,
the conjecture of Umezawa that the elements in \eqref{eq:def-Sk} of
weight $k$ form a basis of $\cZ_k$
(cf.\ \cite[Conjectures~2 and~3]{Umezawa-evaluation}) can now be
stated as follows: conjecturally,
\begin{align*}
 C\colon\fM\overset{\overset{?}{\sim}}{\longrightarrow}\cZ
\end{align*}
is an isomorphism, which would give the expected isomorphism
\eqref{eq:expected-isomorphism}.  One could also define a motivic
version $C^{\mathfrak m}\colon\fM\rightarrow\cH$ of $C$.  For this
map, injectivity and surjectivity are equivalent, since the graded
dimensions of $\fM$ and $\cH$ agree by \eqref{eq:f-alphabet} and
\Cref{prop:zinbiel}, but both are still open.

A similar construction was considered by Chapoton \cite{Chapoton}.
In his work, he constructs free Zinbiel algebras on two generators in
weights $2$ and $3$ inside an algebra of formal iterated integrals and
conjectures that, for generic parameters, these give models for the
motivic multiple zeta values.  In our case, the algebra $\fM$ comes with the evaluation
map $C$, and \Cref{thm:cotangent} shows that its periods are
multiple zeta values.  It would be interesting to understand the
relationship between these two constructions.

\begin{rem}
Set $W_0=xy$, let $\fN_{>0}$ be the Zinbiel subalgebra of $\fH_{>0}$
generated by $x$ and $W_0$, and set $\fN=\QQ\oplus\fN_{>0}$.  The
same leading-word argument as in the proof of \Cref{prop:zinbiel}
shows that $\fN_{>0}$ is free on these two generators.  Moreover, we
have $xx=x\prec x$ and $W_1=W_0\prec x$, and hence $\fM\subset\fN$.
The Hilbert series of $\fN$ is $\frac{1}{1-T-T^2}$, which by
\eqref{eq:level-six-alphabet} is also the Hilbert series of $\cH_6$.
Since $\fN\subset\QQ+x\fH$, the first part of the proof of
\Cref{prop:membership} gives a graded algebra homomorphism
$I^{\mathfrak m}_\gamma\colon\fN\rightarrow\cH_6$, and it is natural
to ask if this is an isomorphism.  So the inclusion $\fM\subset\fN$
is compatible with the inclusion $\cH\subset\cH_6$ under the motivic
iterated integral.
\end{rem}

\begin{rem}
One can also define iterated cotangent integrals of level $n\geq1$ by
replacing the upper bound $\pi/6$ by $\pi/n$.  For $n\geq2$, these are
again defined on $\QQ+x\fH$.  For $n=1$, the words additionally need to
end in $x$.  We leave these investigations to motivated students.
\end{rem}

\section{Some exact formulas}\label{sec:formulas}

In this section, we give some examples in both directions:
shifted log-sine integrals in terms of multiple zeta values, and
multiple zeta values in terms of shifted log-sine integrals.  In depth
one, we have the following classical formula, which appears, for example,
in \cite{Lewin,Choi-Cho-Srivastava} and was used by Umezawa
\cite{Umezawa-evaluation}.

\begin{prop}\label{prop:depth-one}
For odd $k\geq3$, we have
\begin{equation}\label{eq:depth-one}
 \SLs(k)
 =(-1)^{\frac{k-3}{2}}(k-2)!
 \Biggl(
 \Bigl(1-\frac12\bigl(1-2^{1-k}\bigr)\bigl(1-3^{1-k}\bigr)\Bigr)
 \zeta(k)
 +\sum_{j=1}^{\frac{k-3}{2}}
 \frac{(-1)^j\pi^{2j}}{3^{2j}(2j)!}\,
 \zeta(k-2j)
 \Biggr).
\end{equation}
As a consequence, we have
$\zeta(k)\in\operatorname{span}_{\QQ}
\{\pi^{2j}\SLs(k-2j)\mid0\leq j\leq\frac{k-3}{2}\}$.
\end{prop}

\begin{proof}
Formula \eqref{eq:depth-one} is a rearrangement of
\cite[Equation~(4.14)]{Choi-Cho-Srivastava}.  The coefficient of
$\zeta(k)$ is nonzero.  The second statement follows by
induction on $k$.
\end{proof}

For example, we have $\SLs(3)=\frac23\zeta(3)$ and
$\SLs(5)=-\frac{29}{9}\zeta(5)+2\zeta(2)\zeta(3)$.

The shuffle product formula \eqref{eq:sls-shuffle} also gives exact
evaluations in every depth.  For example, for odd $k,l\geq3$ and
$r\geq0$, we have $r!\,\SLs(\{k\}^r)=\SLs(k)^r$, where $\{k\}^r$
denotes $r$ repetitions of $k$, and
$\SLs(k,l)+\SLs(l,k)=\SLs(k)\SLs(l)$.  Together with
\Cref{prop:depth-one}, this gives, e.g.,
\begin{align*}
 \SLs(3,3)
 =\frac29\zeta(3)^2,
 \qquad
 \SLs(3,5)+\SLs(5,3)
 =-\frac{58}{27}\zeta(3)\zeta(5)
 +\frac43\zeta(2)\zeta(3)^2.
\end{align*}
The stuffle product gives formulas in the other direction, e.g.,
$\zeta(3,3)=\frac94\SLs(3,3)-\frac12\zeta(6)$.

For the individual values $\SLs(k,l)$, no explicit formula is known.
The conjectural evaluations in this work were found by computing the
shifted log-sine integrals numerically to more than $200$ decimal digits
and by using the lindep command in Pari/GP.  For example, we find
\begin{align}\label{eq:conj-S35}
 \SLs(3,5)
 \overset{?}{=}
 \frac{159}{80}\zeta(8)
 -\frac{29}{9}\zeta(3)\zeta(5)
 +\zeta(2)\zeta(3)^2
 -\frac{569}{135}\zeta(3,5).
\end{align}
Together with the evaluation of $\SLs(3,5)+\SLs(5,3)$ above, this also
determines $\SLs(5,3)$.  In the other direction, these formulas give
\begin{align*}
 \zeta(3,5)
 &\overset{?}{=}
 \frac{135}{569}\SLs(5,3)
 +\frac{135}{1138}\SLs(3)\SLs(5)
 -\frac{1215}{2276}\zeta(2)\SLs(3)^2
 +\frac{4293}{9104}\zeta(8).
\end{align*}
In depth three, the shuffle product formula and the formulas above do
not determine the values $\SLs(3,3,5)$, $\SLs(3,5,3)$ and
$\SLs(5,3,3)$.  But together with the evaluations in depth two, each
of these values determines the other two.  Numerically, we find
\begin{align*}
 \SLs(3,3,5)
 \overset{?}{=}{}&
 -\frac{4266821}{8640}\zeta(11)
 +\frac{115}{18}\zeta(2)\zeta(9)
 +\frac{1455}{4}\zeta(4)\zeta(7)
 +\frac{4355}{54}\zeta(6)\zeta(5)
 +\frac{159}{80}\zeta(8)\zeta(3)\\
 &-\frac{29}{18}\zeta(3)^2\zeta(5)
 +\frac13\zeta(2)\zeta(3)^3
 +\frac{2347}{162}\zeta(3)\zeta(3,5)
 -\frac{15149}{810}\zeta(3,5,3).
\end{align*}
It would be interesting to find explicit
evaluations of $\SLs(k_1,\dots,k_r)$ in general.\\

\begingroup
\footnotesize

{\bf Use of AI tools.}  In the preparation of this article, the author
used ChatGPT 5.6 Pro and Claude Fable for numerical and exploratory
calculations and to improve the language of the manuscript.  All
results and their proofs were worked out and verified by the author,
who takes full responsibility for the content of this article.

\endgroup

\end{document}